\documentclass[11pt]{article}

\usepackage{mathtools}

\usepackage{float}
\usepackage{amsmath,amsthm}

\usepackage{amssymb}

\usepackage{mathrsfs} 
\usepackage[mathcal]{eucal}

\usepackage[active]{srcltx}

\numberwithin{equation}{section}

\newtheorem{theorem}{Theorem}[section]

\newtheorem{lemma}[theorem]{Lemma}
\newtheorem{proposition}[theorem]{Proposition}

\newcommand\N{\mathbb{N}}

\newcommand{\inner}[2]{\langle#1,#2\rangle}
\newcommand{\Inner}[2]{\left\langle#1,#2\right\rangle}
\newcommand\norm[1]{\|#1\|}

\newcommand\set[1]{\{#1\}}
\newcommand\Set[1]{\left\{#1\right\}}

\newcommand{\tos}{\rightrightarrows}
\newcommand{\tow}{\xrightarrow{w}}  
\newcommand{\tows}{\xrightarrow{w*}}

\newcommand{\br}{\overline{\mathbb{R}}}

\newcommand{\vgap}{\vspace{.1in}}

\title{A weakly convergent fully inexact Douglas-Rachford method with
  relative error tolerance}

\author{Benar F. Svaiter\thanks{IMPA, Estrada Dona Castorina 110,
    22460--320 Rio de Janeiro, Brazil ({\texttt{benar@impa.br}}) 
    tel: 55 (21) 25295112, fax: 55 (21)25124115.}\hspace{.5em}
  \thanks{Partially supported by CNPq
    grant 306247/2015--1
    and by 
    FAPERJ grant Cientistas de Nosso Estado
    E-26/203.318/2017 
  }}

\begin{document}
\maketitle

\begin{abstract}
  Douglas-Rachford method is a splitting algorithm for finding a zero
  of the sum of two maximal monotone operators. Each of its iterations
  requires the sequential solution of two proximal subproblems. The
  aim of this work is to present a fully inexact version of
  Douglas-Rachford method wherein both proximal subproblems are solved
  approximately within a relative error tolerance. We also present a
  semi-inexact variant in which the first subproblem is solved exactly
  and the second one inexactly. We prove that both methods generate
  sequences weakly convergent to the solution of the underlying
  inclusion problem, if any.
 
  \bigskip

  \bigskip


  \bigskip


  \bigskip

\end{abstract}

\section{Introduction}
\label{sec:int}

Douglas-Rachford method~\cite{MR0084194}, originally proposed for
solving discretized heat equations, was extended by Lions and
Mercier~\cite{MR0551319} for finding a zero of the sum of two maximal
monotone operators.  This method is, presently, the subject of intense
research due to its efficiency, its use for solving PDEs, large-scale
optimization problems (even some non-convex ones), imaging problems,
and its connection with the alternating direction method
(see~\cite{gabay,MR2858834,MR3134441,MR3260973,MR3260382,MR3394965,MR3587821,MR3516863,MR3535928}
and the references therein).

Eckstein and Bertseka~\cite{MR1168183} proved that Douglas-Rachford
method can be regarded as an instance of the proximal point method
applied to an implicitly defined operator.  Recently, Eckstein and
Yao~\cite{EckYao} cleverly used this result to propose an inexact
version of Douglas-Rachford method with relative error tolerance based
on Solodov and Svaiter hybrid proximal-extragradient
method~\cite{hpe1,hpe2,hpe-unified,hpe-breg}.  Each iteration of
Douglas-Rachford method requires the sequential solution of two
proximal subproblems.  Eckstein-Yao algorithm is a semi-inexact
version of this method in the sense that it allows for inexactness on
the first proximal subproblem and requires the second one to be solved
exactly.  Complexity of Eckstein-Yao inexact Douglas-Rachford method
was derived by Marques Alves and Geremia~\cite{AlvGer}.

The main contribution of this work is the introduction and analysis of
a fully inexact version of Douglas-Rachford method wherein both
proximal subproblems are solved approximately within a relative error
tolerance.  To the best of our knowledge, this is the first fully
inexact version of Douglas-Rachford method with a relative error
tolerance.  We also propose a semi-inexact version for the cases where
one of the proximal subproblems can be solved exactly.  In the
semi-inexact case of Douglas-Rachford, our proposal is to solve the
first subproblem exactly and the second one inexactly, so that the
error criterion to be satisfied is immediately available (during the
computation of the approximate solution of the second subproblem).

Douglas-Rachford method for solving the inclusion problem
\begin{align*}
  0 \in A(x)+B(x),
\end{align*}
where $A$ and $B$ are maximal monotone operators in a Hilbert space
$H$, can be written as: choose $x_0,b_0\in H$ and for $k=1,2,\dots$
\begin{equation*}
\begin{alignedat}{3}
  &\text{compute }y_k,\;a_k\in H
  &\;\;\text{such that}\quad&a_k\in A(y_k),
  &\quad&\lambda a_k+y_k=x_{k-1}-\lambda b_{k-1},\\  
  &\text{compute }x_k,\;b_k\in H
  &\text{such that}\quad&b_k\in B(x_k),
  &&\lambda b_k+x_k=y_k+\lambda b_{k-1}.
\end{alignedat}
\end{equation*}
Since $x_k=x_{k-1}-\lambda(a_k+b_k)$ and
$b_k=b_{k-1}-\lambda^{-1}(x_k-y_k)$, we propose the following (fully)
inexact version of Douglas-Rachford method: choose $z_0,w_0\in H$ and
for $k=1,2,\dots$
\begin{align*}
\begin{aligned}
  &\text{compute }y_k,\;a_k\in H\text{ and }x_k,\;b_k\in H
  \text{ such that:}
  \\[.2em]
&\begin{alignedat}{3}
  &y_k,a_k \;\;\text{is an approximate solution of }&\;\;&a\in A(y),
  \;\; \lambda a+y={z}_{k-1}-\lambda {w}_{k-1},\\  
  &x_k,b_k\;\;\text{is an approximate solution of }&&b\in B(x),
  \;\;\lambda b+x=y_k+\lambda {w}_{k-1},
\end{alignedat}
  \\[.2em]
  &\text{set }{z}_k={z}_{k-1}-(1-t_k)\lambda (a_k+b_k),\;
  {w}_k={w}_{k-1}-(1-t_k)\lambda^{-1}(x_k-y_k),
\end{aligned}
\end{align*}
where $t_k$ is an under-relaxation parameter which is zero whenever it
is ``safe'' to do so (more of which latter).  If both subproblems are
exactly solved and $t_k=1$, at the $k$-th iteration of this generic
method, then $z_k=x_k$ and $w_k=b_k$.
Following~\cite{hpe2,hpe-unified}, we will allow errors in the
inclusions \emph{and} in the equations in the computation of $y_k,a_k$
and $x_k,b_k$.  Errors in the inclusion will be considered using the
$\varepsilon$-enlargement~\cite{bis} of the maximal monotone operator
$B$.

Since our semi-inexact method is related to~\cite{EckYao}, it is worth
discussing their differences.  Eckstein-Yao semi-inexact version of
Douglas-Rachford method~\cite{EckYao} does not require the
introduction of a relaxation factor, as ours do, so that their inexact
version is formally closer than ours to the exact method.  Their
version allows for errors in the second subproblem, while our
semi-inexact version allows for errors in the second one.  Weak
convergence, in infinite dimensional spaces, of their version is an
interesting open question, while we prove here weak convergence of our
version.  Their version has a very good practical
performance~\cite[Section 7]{EckYao} and nice complexity
properties~\cite{AlvGer}, while the practical performance and
complexity properties of our version are, as of now, open questions.
Weak convergence on Douglas-Rachford method was established by the
author in~\cite{MR2783211} for the inexact version with the summable error
tolerance. Here we use the techniques and ideas of that work to prove
weak convergence under a relative error tolerance.

This work is organized as follows.  In Section \ref{sec:bd} we
establish the notation and prove some basic results.  In Section
\ref{sec:idr} we the present a fully inexact Douglas-Rachford method
and analyze its convergence properties.  In Section \ref{sec:sidr} we
present a semi-inexact Douglas-Rachford method and analyze its
convergence properties.  In Appendices \ref{sec:comp} and
\ref{sec:tech} we prove some technical results.

\section{Basic definitions and results}
\label{sec:bd}

From now on $H$ is a real Hilbert space with inner product
$\inner{\cdot}{\cdot}$ and induced norm
$\norm{x}=\sqrt{\inner{x}{x}}$.  In $H \times H$ we consider the
canonical inner product and associated norm of Hilbert spaces
products.

We are concerned with the inclusion problem
\begin{align}
  \label{eq:p}
  0 \in A(x)+B(x)
\end{align}
where $A:H\tos H$ and $B:H\tos H$ are maximal monotone operators.  The
\emph{extended solution set}~\cite{EckSva} of this problem is
\begin{align}
  \label{eq:se}
  S_e(A,B)=\set{(z,w)\;:\;w\in B(z),\, -w\in A(x)}.
\end{align}
It is trivial to verify that
\begin{align*}
  0 \in A(z)+B(z) \iff
  (z,w)\in S_e(A,B)\text{ for some }w\in H.
\end{align*}

The $\varepsilon$-enlargement~\cite{bis} of a maximal monotone
operator $T:H\tos H$ is defined as
\begin{align}
  \label{eq:teps}
  T^{[\varepsilon]}(z)=\set{v\in H\;:\;
  \inner{z-z'}{v-v'}\geq -\varepsilon
  \;\; \forall z'\in H,\;v'\in T(z')
  }
  \qquad (\varepsilon\geq0,\;\; z\in H).
\end{align}
The following elementary properties of the $\varepsilon$-enlargement
will be used in the sequel.
\begin{proposition}
  \label{pr:teps}
  If $T:H\tos H$ is maximal monotone, then
  \begin{enumerate}
  \item\label{it:0teps} $T^{[\varepsilon=0]}=T$;
  \item\label{it:lteps}
    $\lambda(T^{[\varepsilon]})=(\lambda T)^{[\lambda\varepsilon]}$
    for any $\varepsilon\geq 0$ and $\lambda>0$.
  \end{enumerate}
\end{proposition}

\begin{proof}
  Item~\ref{it:0teps} follows trivially from~\eqref{eq:teps} and the
  maximal monotonicity of $T$ while item~\ref{it:lteps} follows
  directly from \eqref{eq:teps}.
\end{proof}

In each iteration of Douglas-Rachford method, one shall compute
$(I+\lambda T)^{-1}\zeta$, first with $T=A$ and
$\zeta=x_{k-1}-\lambda b_{k-1}$ and them with $T=B$ and
$\zeta= y_k+\lambda b_{k-1}$.  Let $T$ be a maximal monotone operator
in $H$.  Computation of $z=(I+\lambda T)^{-1}\zeta$ can be decoupled
in an inclusion and an equation, which we call the \emph{proximal
  inclusion-equation system}:
\begin{align}
  \label{eq:psys}
  v\in T(z),\qquad \lambda v+z-\zeta=0.
\end{align}
If we allow errors in the inclusion, by means of the
$\varepsilon$-enlargement of $T$, and error in the equation we get
\begin{align*}
  v\in T^{[\varepsilon]}(z),\qquad
  \lambda v+z-\zeta=r
\end{align*}
where $\varepsilon$ is the \emph{error in the inclusion} and $r$ is
the \emph{residual in the equation} at \eqref{eq:psys}.  In some
sense, $\norm{r}^2+2\lambda\varepsilon$ quantifies the overall error
in the inexact solution of \eqref{eq:psys}, due to the next result
proved in~\cite{error-bound}.
\begin{proposition}
  [\mbox{\cite[Corollary 1]{error-bound}}]
  \label{pr:eb}
  Suppose $T:H\tos H$ is maximal monotone, $\lambda<0$,
  and $\zeta\in H$. If
  \begin{align*}
    v^*\in T(z^*),\quad \lambda v^*+z^*-\zeta=0,
    \qquad
    v\in T^{[\varepsilon]}(z),\quad \lambda v+z-\zeta=r
  \end{align*}
  then
  $\norm{\lambda(v^*-v)}^2+\norm{z^*-z}^2 \leq
  \norm{r}^2+2\lambda\varepsilon$.
\end{proposition}

The next lemma will be instrumental in the convergence analysis of the
inexact methods we propose in this work.

\begin{lemma}
  \label{lm:bas}
  Let $z,w\in X$ and $\lambda>0$. Suppose
  \begin{alignat*}{2}
    \lambda a+y & = z-\lambda w+r,\quad  & a&\in A^{[\varepsilon]}(x),\\
    \lambda b+x & = y+\lambda w+s,       & b&\in B^{[\mu]}(y),
  \end{alignat*}
  and define $z'=z-\lambda(a+b)$, $w'=w-\lambda^{-1}(x-y)$.
  For any $(z^*,w^*)\in S_e(A,B)$,
  \begin{multline*}
    \norm{(z^*,\lambda w^*)-(z,\lambda w)}^2
    \geq \norm{(z^*,\lambda w^*)-(z',\lambda w')}^2
    +\norm{\lambda (a+w)}^2\\
    -[\norm{s}^2 +2\lambda(\inner{r}{a+b}+\varepsilon+\mu)].
    \qquad
  \end{multline*}
\end{lemma}

\begin{proof}
  In view of item \ref{it:lteps} on Proposition~\ref{pr:teps}, it
  suffices to prove the lemma for the case $\lambda=1$, which we
  assume now.  In this case
  \begin{align*}
    a+y=z-w+r,\;\; b+x=y+w+s,\;\;z'=z-(a+b),\;\;w'=w-(x-y).
  \end{align*}

  Fix $(z^*,w^*)\in S_e(A,B)$ and let
  \begin{align*}
    \pi=\inner{z^*-z'}{z'-z}+\inner{w^*-w'}{w'-w}.
  \end{align*}
  In view of definition~\eqref{eq:se}, $w^*\in B(z^*)$ and
  $-w^*\in A(z^*)$.  It follows from these inclusions, the inclusions
  $a\in A^{[\varepsilon]}(y)$ and $b\in B^{[\mu]}(x)$, and
  definition~\eqref{eq:teps}, that
  \begin{align*}
    \inner{z^*-z'}{z'-z}
    &=\inner{z^*-z'}{w^*-b}+\inner{z^*-z'}{-w^*-a}
    \\
    &=\inner{z^*-x}{w^*-b}+\inner{x-z'}{w^*-b}
      +\inner{z^*-y}{-w^*-a}+\inner{y-z'}{-w^*-a}
    \\
    &\geq
     \inner{x-z'}{w^*-b}+\inner{y-z'}{-w^*-a}-\varepsilon-\mu
      \,.
  \end{align*}
  Direct combination of this inequality with the definitions of $\pi$
  and $w'$ yields
  \begin{equation*}
    \pi
    \geq
    \inner{x-z'}{w^*-b}+\inner{y-z'}{-w^*-a}-\varepsilon-\mu
    +\inner{w^*-w'}{y-x}.
  \end{equation*}  
  Since the expression at the right hand-side of the above inequality
  does not depend on $w^*$, we can substitute $b$ for $w^*$ is this
  expression and use the definitions of $z'$ and $w'$ to conclude that
  \begin{align*}
    \pi
    & \geq
      \inner{y-z'}{-b-a}+\inner{b-w'}{y-x}
      -\varepsilon-\mu
    \\
    &=\inner{y-x+r+s}{-a-b}+
      \inner{s}{y-x}-\varepsilon-\mu
    \\
    &=\inner{x-y}{a+b}-[
      \inner{r}{a+b}+\inner{s}{a+b+x-y}+
      \varepsilon+\mu].
  \end{align*}
  Therefore,
     \begin{align*}
     \norm{(z^*, w^*)-(z, w)}^2
     &=
       \norm{(z^*, w^*)-(z', w')}^2
       + \norm{(z', w')-(z, w)}^2
       +2\pi
       \\
     &\geq
       \norm{(z^*, w^*)-(z', w')}^2
       +\norm{a+b}^2+\norm{x-y}^2
     \\
     &\qquad \mbox{}
       +2\inner{x-y}{a+b}
       -2[\inner{r}{a+b}+\inner{s}{a+b+y-x}+\varepsilon+\mu]
     \\
     & =
       \norm{(z^*, w^*)-(z', w')}^2
       +\norm{a+b+x-y}^2
     \\
     &\qquad \mbox{}
         -2[\inner{r}{a+b}+\inner{s}{a+b+y-x}+\varepsilon+\mu].
   \end{align*}
   To end the proof, observe that
  \begin{equation*}
    \norm{a+b+x-y}^2-2\inner{s}{a+b+x-y}
    =\norm{a+b+x-y-s}^2-\norm{s}^2=\norm{a+w}^2-\norm{s}^2
  \end{equation*}
  and combine the two above equations.
\end{proof}

\section{A fully inexact Douglas-Rachford method
  with relative error tolerance}
\label{sec:idr}

In this section we present an Inexact Douglas-Rachford method wherein
both proximal subproblems are to be solved within a relative error
tolerance.  We prove that the sequences generated by this method
converge weakly to a point in $S_e(A,B)$, whenever the solution set of
\eqref{eq:p} is non-empty.

\vgap
\vgap

\noindent
\fbox{
\begin{minipage}[h]{6.4 in}
{\bf Algorithm I: inexact DR method with relative error tolerance}.
\begin{itemize}
\item[(0)] Let $z_0,w_0\in H$, $\lambda>0$, $0<\sigma<\nu<1$.
  For $k=1,2,3,\dots$, 
\item[(1)] find $a_k,y_k,\varepsilon_k,
  b_k,x_k,\mu_k$   such that 
  \begin{align}
    \label{eq:s1}
    \begin{split}
    &
    a_k\in A^{[\varepsilon_k]}(y_k),
    \quad   b_k\in B^{[\mu_k]}(x_k),
    \\
    &
    \begin{aligned}\norm{\lambda a_k+y_k-(z_{k-1}-\lambda w_{k-1})}^2
      +\norm{\lambda b_k+x_k-(y_k+\lambda w_{k-1})}^2
      &
      \\
      +2\lambda(\varepsilon_k+\mu_k)
      &
      \leq
      \dfrac{\sigma^2}{4}
      (\norm{a_k+b_k}^2+\norm{x_k-y_k}^2);
    \end{aligned} 
  \end{split}
  \end{align}
\item [(2)] set
  \begin{align}
    \label{eq:s2}
    \begin{split}
      &\delta_k=\norm{\lambda a_k+y_k-(z_{k-1}-\lambda w_{k-1})}^2
      +\norm{\lambda b_k+x_k-(y_k+\lambda w_{k-1})}^2
      +2\lambda(\varepsilon_k+\mu_k),
      \\
      &\rho_k=\norm{\lambda(a_k+b_k)}^2+\norm{x_k-y_k}^2,
      \\
      &t_k=
      \begin{cases}
        0 & \text{if }a_k+b_k=
        x_k-y_k=0\\
        \nu\max\Set{0,
          \sqrt{\dfrac{4\delta_k}{\sigma^2\rho_k}}
          -\dfrac{\norm{\lambda(a_k+w_{k-1})}^2}{\rho_k}
        }&\text{otherwise},
      \end{cases}
      \\[.6em]
      &    z_k=z_{k-1}-(1-t_k)\lambda(a_k+b_k),\quad
      w_k=w_{k-1}-(1-t_k)\lambda^{-1}(x_k-y_k);    
    \end{split}
  \end{align}
\end{itemize}
\noindent
\end{minipage}
}

\vgap
\vgap

We did not specify how to compute $a_k$, $y_k$, $\varepsilon_k$
,$b_k$, $x_k$, and $\mu_k$, which adds generality to the method.  In
Appendix~\ref{sec:comp} we show that under some mild conditions (on
$A$ and $B$) step (2) is computable.  Whenever $t_k=0$,
$z_k=z_{k-1}-\lambda(a_k+b_k)$ and
$w_k=w_{k-1}-\lambda^{-1} (x_k-y_k)$, so that formally we retrieve a
Douglas-Rachford iteration.  In those iterations where $t_k\neq 0$, we
will have an under-relaxed Douglas-Rachford iteration.

To simplify the convergence analysis, let $r_k$ and $s_k$ denote the
residuals in the equations of the proximal inclusion-equations systems
to be solved for $A$ and $B$ at the $k$-th iteration, that is,
\begin{align}
  \label{eq:rsk}
  \begin{aligned}
  r_k&=\lambda a_k+y_k-(z_{k-1}-\lambda w_{k-1}),
  \quad
  s_k=\lambda b_k+x_k-(y_k+\lambda w_{k-1}),
\end{aligned}
  \qquad  (k=1,2,\dots).
\end{align}
With this notation, \eqref{eq:s1} and the first line of \eqref{eq:s2}
writes
\begin{align}
  \label{eq:syn}
  \begin{split}
    \begin{aligned}
      &\begin{alignedat}{2}
        a_k&\in A^{[\varepsilon_k]}(y_k),
          \quad &\lambda a_k+y_k&=z_{k-1}-\lambda w_{k-1}+r_k,
          \\
          b_k&\in B^{[\mu_k]}(x_k),
          &\lambda b_k+x_k&=y_k+\lambda w_{k-1}+s_k,
        \end{alignedat}
        \\
        &\delta_k=\norm{r_k}^2+\norm{s_k}^2
        +2\lambda(\varepsilon_k+\mu_k),
        \qquad\delta_k\leq\dfrac{\sigma^2}{4}\rho_k.
      \end{aligned}
    \end{split}
\end{align}
It follows from the definition of $t_k$ at \eqref{eq:s2} that for all
$k$,

\begin{alignat}{1}
  \label{eq:boundt}
  0\leq t_k\leq \nu<1,
  \qquad \dfrac{\nu}{\sigma}\sqrt{4\delta_k\rho_k}
  -\nu\norm{\lambda(a_k+w_{k-1})}^2
  \leq t_k\rho_k
  ,\qquad
  t_k^2\rho_k\leq\dfrac{4\nu^2}{\sigma^2}\delta_k.
\end{alignat}

From now on in this section,
\begin{align}
  \label{eq:pk}
  \qquad p_k
    =(z_k,\lambda w_k),\qquad
     \qquad \qquad\qquad  (k=0,1,2,\dots).
\end{align}
First we will prove that the sequence $\set{p_k}$ converges F\'ejer to
the set of points $(z,\lambda w)$ where $(z,w) \in S_e(A,B)$.

\begin{lemma}
  \label{lm:fej}
  For any $({z},{w})\in S_e(A,B)$ and all $k$
  \begin{align*}
    \norm{({z},\lambda {w})-p_{k-1}}^2
    & \geq \norm{({z},\lambda {w})-p_k}^2
    +(1-t_k)
    \left\{\dfrac{\nu-\sigma}{\sigma}\sqrt{4\delta_k\rho_k}
    +
    (1-\nu)\norm{\lambda(a_k+w_{k-1})}^2
    \right\},
    \\
    \norm{({z},\lambda {w})-p_0}^2
    & \geq \norm{({z},\lambda {w})-p_k}^2
    +\sum_{i=1}^k(1-t_i)
    \left\{\dfrac{\nu-\sigma}{\sigma}\sqrt{4\delta_i\rho_i}
    +
    (1-\nu)\norm{\lambda(a_i+w_{i-1})}^2
    \right\}.
  \end{align*}
\end{lemma}

\begin{proof}
  Fix $({z},{w})\in S_e(A,B)$ and let $p^*=({z},\lambda {w})$.
  Define, for $k=1,2,\dots$,
  \begin{align}
    \label{eq:p'k}
    z'_k=z_{k-1}-\lambda(a_k+b_k),
    \quad
    w'_k=w_{k-1}-\lambda^{-1}(x_k-y_k),
    \quad
    p'_k=(z'_k,\lambda w'_k).
  \end{align}
  It follows from \eqref{eq:pk}, \eqref{eq:syn} and Lemma~\ref{lm:bas}
  that
  \begin{align*}
    \norm{p^*-p_{k-1}}^2
    &\geq \norm{p^*-p'_k}^2+\norm{\lambda(a_k+w_{k-1})}^2
      -[\norm{s_k}^2+2\lambda(\inner{r_k}{a_k+b_k}+\varepsilon_k+\mu_k)]
      .
  \end{align*}
  Using Cauchy-Schwarz inequality, the definition of $\rho_k$ in
  \eqref{eq:s2}, and the third line on \eqref{eq:syn} we conclude that
  \begin{align*}
    \norm{s_k}^2+2(\inner{r_k}{a_k+b_k}+\varepsilon_k+\mu_k)
    & \leq
      \norm{s_k}^2+2\lambda(\varepsilon_k+\mu_k)
      +2\norm{r_k}\sqrt{\rho_k}\\
    &=\delta_k-\norm{r_k}^2+2\norm{r_k}
      \sqrt{\rho_k}
  \end{align*}
  and $\norm{r_k}\leq \sqrt{\delta_k}\leq\sigma\sqrt{\rho_k}/2$.
  Since the expression on the right hand-side of the above equality is
  increasing for $\norm{r_k}\leq\sqrt{\rho_k}$, its maximum value on
  $[0,\sqrt{\delta_k}]$ is attained at $\norm{r_k}=\sqrt{\delta_k}$.
  Combining these observations with the two above inequalities we
  conclude that
  \begin{align*}
    \norm{p^*-p_{k-1}}^2
    &\geq \norm{p^*-p'_k}^2+\norm{\lambda(a_k+w_{k-1})}^2
      -2\sqrt{\delta_k\rho_k}.
  \end{align*}
   Since $p_k=t_kp_{k-1}+(1-t_k)p'_k$ and $\norm{p_k-p'_k}^2=\rho_k$,
   \begin{align*}
     \norm{p^*-p_{k-1}}^2
     &=t_k\norm{p^*-p_{k-1}}+(1-t_k)\norm{p^*-p_{k-1}}^2
     \\
     &
       \geq t_k\norm{p^*-p_{k-1}}+
       (1-t_k)\left[
       \norm{p^*-p'_k}^2+\norm{\lambda(a_k+w_{k-1})}^2
       -2\sqrt{\delta_k\rho_k}\right]
     \\
     &
       =
       \norm{t_k(p^*-p_{k-1})+(1-t_k)(p^*-p'_k)}^2+
      (1-t_k)t_k\norm{p_{k-1}-p'_k}^2\\
     &\quad\mbox{}+
       (1-t_k)\left[\norm{\lambda(a_k+w_{k-1})}^2
       -2\sqrt{\delta_k\rho_k}\right]
     \\
     &
       =
       \norm{p^*-p_k}+
       (1-t_k)\left[t_k\rho_k+\norm{\lambda(a_k+w_{k-1})}^2
       -2\sqrt{\delta_k\rho_k}\right]
       .
   \end{align*}
   To end the proof of the first inequality, use the above inequality
   and the next to last inequality in \eqref{eq:boundt}.  The second
   inequality of the lemma follows trivially from the first one.
 \end{proof} 

 \begin{lemma}
   \label{lm:prec}
   If $S_e(A,B)$ is nonempty, then
   \begin{enumerate}
   \item \label{it:bd}
     $\set{z_k}$ and $\set{w_k}$ are bounded;
   \item \label{it:rsem}
     $r_k\to 0$, $s_k\to 0$, $\varepsilon_k\to 0$ and $\mu_k \to 0$ 
     as $k\to \infty$;
   \item \label{it:ay}
     $a_k+w_{k-1}\to 0$ and $y_k-z_{k-1}\to 0$ as $k\to \infty$;
   \item \label{it:bx}
     $b_k-w_k \to 0$ and $x_k-z_k \to 0$  as $k\to \infty$.
   \end{enumerate}
 \end{lemma}

 \begin{proof}
 Take $({z},{w})$ in $S_e(A,B)$. It follows from Lemma~\ref{lm:fej} that
 $\set{(z_k,\lambda w_k)}$ is bounded, which proves item~\ref{it:bd}, and
 that
 \begin{align*}
   \sum_{k=1}^\infty\sqrt{\delta_k\rho_k} < \infty,\qquad
   \sum_{k=1}^\infty \norm{a_k+w_{i-1}}^2 <\infty.
 \end{align*}
 Since $\delta_k\leq\sigma^2\rho_k/4$,
 $\sqrt{\delta_k\rho_k}\geq 2 \delta_k/\sigma$ and it follows from the
 first above inequality that $\delta_k\to 0$ as $k\to \infty$. This
 result, together with the third line of \eqref{eq:syn} proves item
 \ref{it:rsem}.  The first limit in item~\ref{it:ay} follows trivially
 form the second above inequality while the second limit follows from
 the first one, the first limit in item~\ref{it:rsem} and the
 definition of $r_k$ in \eqref{eq:rsk}.

 It follows from the update formula for $z_k$ and $w_k$, and from
 \eqref{eq:syn} that
 \begin{align*}
   z_k-x_k+r_k+s_k=t_k\lambda(a_k+b_k),\qquad
   w_k-b_k+s_k=t_k(x_k-y_k).
 \end{align*}
 Squaring both sides of each one of the above equations and adding them
 we conclude that
 \begin{align*}
   \norm{z_k-x_k+r_k+s_k}^2+\norm{w_k-b_k+s_k}^2
   =t_k^2\rho_k.
 \end{align*}
 Since $\delta_k\to 0$ as $k\to \infty$, it follows from the above
 equations and the last inequality in~\eqref{eq:boundt} that 
 \begin{align*}
   t_k^2\rho_k\to 0,\quad   z_k-x_k+r_k+s_k\to 0,\quad
   w_k-b_k+s_k\to 0\qquad \text{as }k\to \infty.
 \end{align*}
 Item \ref{it:bx} follows from the above result and from item~\ref{it:rsem}.
\end{proof}

\begin{theorem}
  If $S_e(A,B)\neq \emptyset$, then $\set{(x_k,b_k)}$,
  $\set{(y_k,-a_k)}$ and $\set{(z_k,w_k)}$ converge weakly to a point
  in this set.
\end{theorem}

\begin{proof}
  Suppose a subsequence $\set{(z_{k_n},w_{k_n})}$ converges weakly to
  some $(z,w)$.  It follows from Lemma~\ref{lm:prec} that
  \begin{alignat*}{4}
    &x_{k_n}\tow  z, &\quad& y_{k_n-1}\tow z,\quad
    &&x_{k_n}-y_{k_n-1}\to 0,
    \\
    &b_{k_n}\tow w, && a_{k_n-1}\tow -w,&\quad& a_{k_n-1}+b_{k_n}\to 0,
    &
    \\
    &\mu_{k_n}\to 0,&&\varepsilon_{k_n-1}\to 0,&&
    &\qquad \text{as }k\to \infty.
  \end{alignat*}
  Since $b_{k_n}\in B^{[\mu_{k_n}]}(x_{k_n})$,
  $a_{k_n-1}\in A^{[\varepsilon_{k_n-1}]}(y_{k_n-1})$
  for $n=1,2,\dots$, it follows from Lemma~\ref{lm:limit} that
  $w\in B(z)$ and $-w\in A(w)$. Therefore, $(z,w)\in S_e(A,B)$.

  We have proved that all weak limit points of $\set{(z_k,w_k)}$ are
  in $S_e(A,B)$.  Therefore, all weak limit points of
  $\set{p_k=(z_k,\lambda w_k)}$ are in $\Omega$,
  \begin{align*}
    \Omega=\set{(z,\lambda w)\;:\;(z,w)\in S_e(A,B)}.
  \end{align*}
  In Lemma~\ref{lm:fej} we proved that $\set{p_k}$ is Fej\'er
  convergent to $\Omega$.  Since $\Omega\neq\emptyset$, it follows
  from these results and from Opial's Lemma~\cite{opial} that the
  bounded sequence $\set{p_k}$ has a unique weak limit point and such
  a point belongs to $\Omega$.  Therefore, $\set{p_k}$ converges
  weakly to a point in $(z,\lambda w)$ where $(z,w)\in S_e(A,B)$.  To
  end the proof, use items~\ref{it:ay} and \ref{it:bx} of
  Lemma~\ref{lm:prec}.
\end{proof}

\section{A semi-inexact Douglas-Rachford method}
\label{sec:sidr}

In this section we present an inexact Douglas-Rachford method wherein,
in each iteration, the first proximal subproblem is to be exactly
solved and the second proximal subproblem is to be solved within a
relative error tolerance.

A possible advantage of solving the second subproblem approximately,
instead of the first one, is that the error criterion to be satisfied
is readily available during the computation of the second step,
thereby obviating the necessity of computing more than once the
approximate solution per iteration.  Since one of the subproblems is
to be solved exactly, following~\cite{EckYao}, we call this method
semi-inexact.

\vgap
\vgap

\noindent
\fbox{
  \begin{minipage}[h]{6.4 in} {\bf Algorithm II: A semi-inexact
      Douglas-Rachford method with relative error tolerance}.
\begin{itemize}
\item[(0)] Let $z_0,w_0\in H$, $\lambda>0$, $0<\sigma<\nu<1$.
  For $k=1,2,3,\dots$,
\item[(1)] compute $a_k,y_k$
  such that 
  \begin{subequations}
    \begin{align}
      \label{eq:s1a}
      &a_k\in A(y_k),\quad
        \lambda a_k + y_k = z_{k-1} - \lambda w_{k-1},
        \intertext{compute  $b_k,x_k,\mu_k$
        such that}
        \label{eq:s1b}
      &b_k\in B^{[\mu_k]}(x_k),\;
        \norm{\lambda b_k+x_k-(y_k+\lambda w_{k-1})}^2
        +2\lambda\mu_k\leq  \sigma^2
        (\norm{a_k+b_k}^2
        +\norm{x_k-y_k}^2);
  \end{align}
\end{subequations}
\item [(2)] set
  \begin{align}
    \label{eq:s2.si}
    \begin{split}
      &\delta_k=
      \norm{\lambda b_k+x_k-(y_k+\lambda w_{k-1})}^2
      +2\lambda \mu_k,
      \\
      &\rho_k=\norm{\lambda(a_k+b_k)}^2+\norm{x_k-y_k}^2,
      \\
      &t_k=
      \begin{cases}
        0 & \text{if }a_k+b_k=
        x_k-y_k=0\\
        \nu^2\max\Set{0,
          \dfrac{\delta_k}{\sigma^2\rho_k}
          -\dfrac{\norm{\lambda(a_k+w_{k-1})}^2}{\rho_k}
        }&\text{otherwise},
      \end{cases}
      \\[.6em]
      &    z_k=z_{k-1}-(1-t_k)\lambda(a_k+b_k),\quad
      w_k=w_{k-1}-(1-t_k)\lambda^{-1}(x_k-y_k);    
    \end{split}
  \end{align}
\end{itemize}
\noindent
\end{minipage}
}

\vgap
\vgap

From now on in this section, $\set{z_k}$, $\set{w_k}$, $\set{a_k}$
etc. are sequences generated by Algorithm {\bf II}.  To simplify the
convergence analysis, let $s_k$ denote the residuals in the equation
of the inclusion-equation system to be solved for $B$ at the $k$-th
iteration, that is,
\begin{align}
  \label{eq:rsk.si}
  s_k=\lambda b_k+x_k-(y_k+\lambda w_{k-1}),
  \qquad\qquad (k=1,2,\dots).
\end{align}
With this notation, \eqref{eq:s1a}, \eqref{eq:s1b} and the first line
of \eqref{eq:s2.si} writes
\begin{align}
  \label{eq:syn.si}
  \begin{split}
    \begin{aligned}
      &\begin{alignedat}{2}
        a_k&\in A(y_k),
          &\lambda a_k+y_k&=z_{k-1}-\lambda w_{k-1},
          \\
          b_k&\in B^{[\mu_k]}(x_k),
          \;\;\;
          &\lambda b_k+x_k&=y_k+\lambda w_{k-1}+s_k,
        \end{alignedat}
        \\
        &\delta_k=\norm{s_k}^2
        +2\lambda\mu_k,
        \qquad\delta_k\leq\sigma^2\rho_k.
      \end{aligned}
    \end{split}
\end{align}
It follows from the definition of $t_k$ at \eqref{eq:s2.si} that for
all $k$,
\begin{alignat}{1}
  \label{eq:boundt.si}
  0\leq t_k\leq \nu^2<1,
  \qquad \dfrac{\nu^2}{\sigma^2}\delta_k-\nu^2\norm{\lambda(a_k+w_{k-1})}^2
  \leq t_k\rho_k
  ,\qquad t_k^2\rho_k \leq \dfrac{\nu^4}{\sigma^2}\delta_k.
\end{alignat}

From now on in this section,
\begin{align}
  \label{eq:pk.si}
  p_k
    =(z_k,\lambda w_k),\qquad\qquad
  (k=0,1,2,\dots).
\end{align}
We will prove that the sequence $\set{p_k}$ converges to a point
$(z,\lambda w)$ where $(z,w) \in S_e(A,B)$.

\begin{lemma}
  \label{lm:fej.si}
  For any $(z^*,w^*)\in S_e(A,B)$ and all $k$
  \begin{align*}
    \norm{(z^*,\lambda w^*)-p_{k-1}}^2
    & \geq \norm{(z^*,\lambda w^*)-p_k}^2
    +(1-t_k)
    \left\{\dfrac{\nu^2-\sigma^2}{\sigma^2}\delta_k
    +
    (1-\nu^2)\norm{\lambda(a_k+w_{k-1})}^2
    \right\},
    \\
    \norm{(z^*,\lambda w^*)-p_0}^2
    & \geq \norm{(z^*,\lambda w^*)-p_k}^2
    +\sum_{i=1}^k(1-t_i)
    \left\{\dfrac{\nu^2-\sigma^2}{\sigma^2}\delta_i
    +
    (1-\nu^2)\norm{\lambda(a_i+w_{i-1})}^2
    \right\}.
  \end{align*}
\end{lemma}

\begin{proof}
  Fix $(z^*,w^*)\in S_e(A,B)$ and let $p^*=(z^*,\lambda w^*)$.
  Define, for $k=1,2,\dots$,
  \begin{align}
    \label{eq:p'k.si}
    z'_k=z_{k-1}-\lambda(a_k+b_k),
    \quad
    w'_k=w_{k-1}-\lambda^{-1}(x_k-y_k),
    \quad
    p'_k=(z'_k,\lambda w'_k).
  \end{align}
  It follows from \eqref{eq:pk.si}, \eqref{eq:syn.si} and
  Lemma~\ref{lm:bas} that
  \begin{align*}
    \norm{p^*-p_{k-1}}^2
    &\geq \norm{p^*-p'_k}^2+\norm{\lambda(a_k+w_{k-1})}^2
      -[\norm{s_k}^2+2\lambda\mu_k]
    \\
      &\geq \norm{p^*-p'_k}^2+\norm{\lambda(a_k+w_{k-1})}^2
      -\sigma^2\rho_k.
\end{align*}
   Since $p_k=t_kp_{k-1}+(1-t_k)p'_k$ and $\norm{p_k-p'_k}^2=\rho_k$,
   \begin{align*}
     \norm{p^*-p_{k-1}}^2
     &=t_k\norm{p^*-p_{k-1}}+(1-t_k)\norm{p^*-p_{k-1}}^2
     \\
     &
       \geq t_k\norm{p^*-p_{k-1}}+
       (1-t_k)\left[
       \norm{p^*-p'_k}^2+\norm{\lambda(a_k+w_{k-1})}^2
       -\delta_k\right]
     \\
     &
       =
       \norm{t_k(p^*-p_{k-1})+(1-t_k)(p^*-p'_k)}^2+
      (1-t_k)t_k\norm{p_{k-1}-p'_k}^2\\
     &\quad\mbox{}+
       (1-t_k)\left[\norm{\lambda(a_k+w_{k-1})}^2
       -\delta_k\right]
     \\
     &
       =
       \norm{p^*-p_k}+
       (1-t_k)\left[t_k\rho_k+\norm{\lambda(a_k+w_{k-1})}^2
       -\delta_k\right].
   \end{align*}
   To end the proof of the first inequality, use the above inequality
   and the next to last inequality in \eqref{eq:boundt.si}.

   The second inequality of the lemma follows trivially from the first
   one.
 \end{proof} 

 \begin{lemma}
   \label{lm:prec.si}
   If $S_e(A,B)$ is nonempty, then
   \begin{enumerate}
   \item \label{it:bd.si}
     $\set{z_k}$ and $\set{w_k}$ are bounded;
   \item \label{it:rsem.si}
     $s_k\to 0$ and $\mu_k \to 0$ 
     as $k\to \infty$;
   \item \label{it:ay.si}
     $a_k+w_{k-1}\to 0$ and $y_k-z_{k-1}\to 0$ as $k\to \infty$;
   \item \label{it:bx.si}
     $b_k-w_k \to 0$ and $x_k-z_k \to 0$  as $k\to \infty$.
   \end{enumerate}
 \end{lemma}

 \begin{proof}
   Take $(z^*,w^*)$ in $S_e(A,B)$. It follows from
   Lemma~\ref{lm:fej.si} that $\set{(z_k,\lambda w_k)}$ is bounded,
   which proves item~\ref{it:bd.si}, and that
 \begin{align*}
  \sum_{k=1}^\infty\delta_k < \infty,\qquad
  \sum_{k=1}^\infty \norm{a_k+w_{i-1}}^2 <\infty.
 \end{align*}
 It follows from the first above inequality that $\delta_k\to 0$ as
 $k\to \infty$.  This result, together with the third line of
 \eqref{eq:syn.si} proves item \ref{it:rsem.si}.  Item~\ref{it:ay.si}
 follows trivially form the second above inequality and the equality
 in \eqref{eq:s1a}.

 It follows from the update formulas for $z_k$ and $w_k$, and from
 \eqref{eq:syn.si} that
 \begin{align*}
   z_k-x_k+s_k=t_k\lambda(a_k+b_k),\qquad
   w_k-b_k+s_k=t_k(x_k-y_k).
 \end{align*}
 Squaring both sides of each one of these equalities and adding them
 we conclude that
 \begin{align*}
   \norm{z_k-x_k+s_k}^2+\norm{w_k-b_k+s_k}^2
   =t_k^2\rho_k.
 \end{align*}
 Since $\delta_k\to 0$ as $k\to \infty$, it follows from the above
 equations and the last inequality in~\eqref{eq:boundt.si} that 
 \begin{align*}
   t_k^2\rho_k\to 0,\quad   z_k-x_k+s_k\to 0,\quad
   w_k-b_k+s_k\to 0\qquad \text{as }k\to \infty.
 \end{align*}
 Item \ref{it:bx.si} follows from the above result and from
 item~\ref{it:rsem.si}.
\end{proof}

\begin{theorem}
  If $S_e(A,B)\neq \emptyset$, then
  $\set{(x_k,b_k)}$, $\set{(y_k,-a_k)}$ and $\set{(z_k,w_k)}$ converge
  weakly to a point in this set.
\end{theorem}

\begin{proof}
  Suppose a subsequence $\set{(z_{k_n},w_{k_n})}$ converges
  weakly to some $(z,w)$.  It follows from
  Lemma~\ref{lm:prec.si} that
  \begin{align*}
  \begin{alignedat}{3}
    &x_{k_n}\tow  z, &\quad& y_{k_n-1}\tow z,\quad
    &&x_{k_n}-y_{k_n-1}\to 0,
    \\
    &b_{k_n}\tow w, && a_{k_n-1}\tow -w,&\quad& a_{k_n-1}+b_{k_n}\to 0,
    \\
    &\mu_{k_n}\to 0&&\varepsilon_{k_n-1}\to 0&&\qquad
    \end{alignedat}
  \qquad \text{as }k\to \infty.
\end{align*}
  Since $b_{k_n}\in B^{[\mu_{k_n}]}(x_{k_n})$,
  $a_{k_n-1}\in A(y_{k_n-1})$
  for $n=1,2,\dots$, it follows from Lemma~\ref{lm:limit} that
  $w\in B(z)$ and $-w\in A(w)$. Therefore, $(z,w)\in S_e(A,B)$.

  Define, again,
  \begin{align*}
    \Omega=\set{(z,\lambda w)\;:\;(z,w)\in S_e(A,B)}.
  \end{align*}
  We have proved that all weak limit points of $\set{(z_k,w_k)}$ are
  in $S_e(A,B)$.  Therefore, all weak limit points of
  $\set{p_k=(z_k,\lambda w_k)}$ are in $\Omega$.  In
  Lemma~\ref{lm:fej.si} we proved that $\set{p_k}$ is Fej\'er
  convergent to $\Omega$.  Since $\Omega\neq\emptyset$, it follows
  from these results and from Opial's Lemma~\cite{opial} that the
  bounded sequence $\set{p_k}$ has a unique weak limit point and such
  a point belongs to $\Omega$.  Therefore, $\set{p_k}$ converges
  weakly to a point in $(z,\lambda w)$ where $(z,w)\in S_e(A,B)$.  To
  end the proof, use items~\ref{it:ay.si} and \ref{it:bx.si} of
  Lemma~\ref{lm:prec.si}.
\end{proof}

\appendix

\section{Computability of step (2)}\label{sec:comp}

\begin{proposition}
  \label{pr:comp}
  Suppose that $T=A$ and $T=B$ have the following properties
  \begin{enumerate}
  \item for any $v,z\in H$, on can verify whether $v\in T(z)$ or
    $v\notin T(z)$;
  \item for any $c\in H$, on can generate sequences $v_i,z_i,\eta_i$ such
    that $v_i\in T^{[\eta_i]}(z_i)$ for all $i$ and
    \begin{align}
      \eta_i\to 0,\;
      \norm{\lambda v_i+y_i-c}\to 0\qquad
      \text{as }i\to \infty.
    \end{align}
  \end{enumerate}
  Then, step (1) of Algorithms {\bf I} and {\bf II} are computable.
\end{proposition}

\begin{proof}
  It suffices to consider one iteration of Algorithm {\bf I}.
  Assume that we are at iteration $k$ of Algorithm {\bf I}.
  If $-w_{k-1}\in A(z_{k-1})$ and $w_{k-1}\in B(z_{k-1})$
  then
  \begin{align*}
    (a_k,y_k,\varepsilon_k)=(-w_{k-1},z_{k-1},0),\qquad
    (b_k,x_k,\mu_k)=(w_{k-1},z_{k-1},0)
  \end{align*}
  trivially satisfies criterion~\eqref{eq:s1}.

  Suppose $-w_{k-1}\notin A(z_{k-1})$ or $w_{k-1}\notin B(z_{k-1})$
  and let
  \begin{alignat*}{2}
    \hat{y}&=(I+\lambda A)^{-1}(z_{k-1}-\lambda w_{k-1}),&\quad
    \hat a&=\lambda^{-1}(z_{k-1}-\lambda w_{k-1}-\hat{y}),
    \\
    \hat{x}&=(I+\lambda B)^{-1}(\hat{y}+\lambda w_{k-1}),&\quad
    \hat b&=\lambda^{-1}(\hat{y}+\lambda w_{k-1}-\hat{x}).
  \end{alignat*}
  It follows from these definitions that
    $\hat a\in A(\hat y)$, $\hat b\in B(\hat x)$, and
    \begin{equation*}
    \lambda \hat a+\hat y=z_{k-1}-\lambda w_{k-1},\qquad
     \lambda \hat b+\hat x=\hat y+\lambda w_{k-1}.
   \end{equation*}
  If $\hat a+\hat b=\hat x-\hat y=0$, then it follows from the  above
  equalities that
  $w_{k-1}=\hat b=-\hat a$,
  $z_{k-1}=\hat x=\hat y$ and, consequently,
  $-w_{k-1}\in A(z_{k-1})$ and
  $w_{k-1}\in B(z_{k-1})$, in contradiction with our assumption.  
  Therefore, $\hat a+\hat b\neq 0$ or $\hat x-\hat y\neq 0$.
  
  In view of the assumption of the proposition, one can generate sequences
  $\set{(a_{k,j},y_j,\varepsilon_{k,j})}_{j\in \N}$ and
  $\set{(b_{k,j},x_{k,j},\mu_{k,j})}$ such that
  \begin{align*}
    a_{k,j}\in A^{[\varepsilon_{k,j}]}(y_{k,j}),
    \varepsilon_{k,j}\leq\dfrac{1}{2\lambda j^2},
    \norm{\lambda a_{k,j}+y_{k,j}-(z_{k-1}+\lambda w_{k-1})}
    \leq \dfrac{1}{j}
    \\
        b_{k,j}\in B^{[\mu_{k,j}]}(x_{k,j}),
    \mu_{k,j}\leq\dfrac{1}{2\lambda j^2},
    \norm{\lambda b_{k,j}+x_{k,j}-(y_{k-1}+\lambda w_{k-1})}
    \leq \dfrac{1}{j}
  \end{align*}
  It follows from the above relations and from Proposition~\ref{pr:eb}
  that
  \begin{equation*}
    \norm{\lambda (a_{k_j}-\hat a)}^2+\norm{y_{k_j}-\hat y}^2
    \leq \dfrac{2}{j^2}.
  \end{equation*}
  In particular $\norm{y_{k,j}-\hat y}\leq 2/j$ and
  \begin{equation*}
    \norm{\lambda b_{k,j}+x_{k,j}-(\hat y+\lambda w_{k-1})}\leq\dfrac{3}{j}
    .
  \end{equation*}
  Using again~Proposition~\ref{pr:eb} we conclude that
  \begin{equation*}
    \norm{b_{k,j}-\hat b}^2+\norm{x_{k,j}-\hat x}\leq\dfrac{10}{j^2}
  \end{equation*}
  Therefore,
  \begin{align*}
    &y_{k,j}\to \hat y,\quad a_{k,j}\to \hat a,\\
    &x_{k,j}\to \hat x,\quad b_{k,j}\to \hat b\qquad \text{as }
      j\to\infty.
  \end{align*}
  and $\norm{\lambda (a_{k,j}-b_{k,j})}^2+\norm{x_{k,j}-y_{k,j}}^2
  \to \norm{\lambda(\hat a+\hat b)}^2+\norm{\hat x-\hat y}^2> 0$
  as $j\to\infty$.
  Since
  \begin{align*}
  \norm{\lambda a_{k,j}+y_{k,j}-(z_{k-1}+\lambda w_{k-1})}^2
  +
  \norm{\lambda b_{k,j}+x_{k,j}-(y_{k-1}+\lambda w_{k-1})}
    +2\lambda(\varepsilon_{k,j}+\mu_{k,j})
    \leq \dfrac{4}{j^2},
  \end{align*}
  for $j$ large enough criterion~\eqref{eq:s1} will be satisfied.
\end{proof}

\section{A technical Lemma}
\label{sec:tech}

Let $X$ be a real Banach space with topological dual $X^*$
and let $\inner{x}{x^*}$ stands for the duality product
$x^*(x)$ for $x\in X$ and $x^*\in X^*$.
A point-to-set operator $T:X\tos X^*$ is monotone if
\begin{align*}
  \inner{x-y}{x^*-y^*}\geq 0\qquad \forall x, y\in X, \,
  x^*\in T(x), \, y^*\in T(y).
\end{align*}
A monotone operator $T$ is maximal monotone if it is monotone and its
graph is maximal in the family of the graphs of monotone operators.

The \emph{$\varepsilon$-enlargement} of a maximal monotone operator
$T$ in $X$ is defined as
\begin{align*}
  T^{[\varepsilon]}(x)=\set{
  x^*\in X^*\;:\;
  \inner{x-y}{x^*-y^*}\geq -\varepsilon
  }\qquad
  \qquad \varepsilon\geq 0,\;x\in X.
\end{align*}
\emph{Fitzpatrick}~\cite{MR1009594} function $\varphi$ associated with
a maximal monotone operator $T:X\tos X^*$ is defined as
\begin{align}
  \label{eq:fitz}
  \varphi:X \times X^*\to\br, \quad
  \varphi(x,x^*)
  & =\sup_{y^*\in T(y)}
    \inner{x}{y^*}+\inner{y}{x^*}-\inner{y}{y^*}.
\end{align}

\begin{lemma}
  \label{lm:fitz}
  Suppose $T$ is a maximal monotone operator in $X$ ad let
  $\varphi$ be its Fitzpatrick function. Then
  \begin{enumerate}
  \item\label{it:f1}
    $\varphi$ is convex and lower semicontinuous in
    the weak$\times$weak-$*$ topology of $X\times X^*$;
  \item\label{it:f2}
    $\inner{x}{v}\leq \varphi(x,v)$;
  \item\label{it:f3}
    $v\in T(x)\iff \inner{x}{v}=\varphi(x,v)$;
  \item\label{it:f4} $v\in T^{[\varepsilon]}(x)\iff \varphi(x,v)\leq
    \inner{x}{v}+\varepsilon$;
  \end{enumerate}
\end{lemma}

\begin{proof}
  Items \ref{it:f1}, \ref{it:f2}, and \ref{it:f3}
  where proved in~\cite{MR1009594}.

  Item~\ref{it:f4} was proved in \cite{MR1802238}. For the sake of completeness
  we present a proof.
  If follows from~\eqref{eq:fitz} that
  \begin{align*}
    \varphi(x,x^*)-\inner{x}{x^*}=\sup_{y^*\in T(y)}
                 \inner{x-y}{y^*-x^*}.
  \end{align*}
  To end the proof, use \eqref{eq:teps} to write
  \begin{align*}
    x^*\in T^{[\varepsilon]}(x)\iff \varepsilon \geq
    \inner{x-y}{y^*-x^*}\;\;\forall y\in X,\;y^*\in T(y)
  \end{align*}
  and combine the two above equations.
\end{proof}

The next technical lemma will be used in the convergence analysis of the
inexact Douglas-Rachford method proposed in this work.

\begin{lemma}
  \label{lm:limit}
  Let $X$ be a real Banach space.
  If $T_1,\dots,T_m:X\tos X^*$ are maximal monotone
  operators,  $    v_{i,k}\in T^{[\varepsilon_{i,k}]}(x_{i,k})$
  for $i=1,\dots,m$ and $k=1,2,\dots$,
  and
  \begin{alignat*}{2}
    &x_{i,k}-x_{j,k} \to 0\;\; \text{(for $i,j=1,\dots,m$)},
    \quad
    &&\sum_{i=1}^mv_{i,k}\to v,
    \\
    &x_{i,k}\tow
    x,
    \qquad\qquad
    v_{i,k}\tows v_i,
    \quad
    &&\varepsilon_{i,k}\to 0
    \;\; \text{(for $i=1,\dots,m$)},
  \end{alignat*}
  as $k\to \infty$,
  then $v_i\in T_i(x)$ for $i=1\dots,m$.
\end{lemma}

\begin{proof}
  Let $\varphi_i$ be Fitzpatrick's function of $T_i$ for $i=1\dots,m$.
  Since $v_{i,k}\in T^{[\varepsilon_{i,k}]}(x_{i,k})$,
  \begin{align*}
    \varphi_i(x_{i,k},v_{i,k})
    \leq \inner{x_{i,k}}{v_{i,k}}
    +\varepsilon_{i,k}
    \qquad
    i=1,\dots,m.
  \end{align*}
  Adding these inequalities for $i=1,\dots,m$ we obtain, after
  trivial algebraic manipulations the inequality
  \begin{align*}
    \sum_{i=1}^m
    \varphi_i(x_{i,k},v_{i,k})
    \leq
    \inner{x_{1,k}}{v}
    +\Inner{x_{1,k}}{\sum_{i=1}^mv_{i,k}-v}
    +\sum_{i=1}^m
    \inner{x_{i,k}-x_{1,k}}{v_{i,k}}
    +\varepsilon_{i,k}\,.
  \end{align*}
  Each $\varphi_i$ is lower semicontinuous in the weak$\times$weak-$*$
  topology of $X\times X^*$.  Moreover, it follows trivially from the
  assumptions of the lemma that the sequences $\set{x_{i,k}}$ and
  $\set{v_{i,k}}$ are bounded for $i=1,\dots,m$.  Therefore, taking
  the $\liminf$ as $k\to\infty$ at both sides of the above inequality
  and using the lower semicontinuity of $\varphi$ we conclude that
  \begin{align*}
    \sum_{i=1}^m
    \varphi_i(x,v_i)\leq \inner{x}{v}.
  \end{align*}
  It also follows trivially from the assumptions of the lemma that
  $\sum_{i=1}^mv_i=v$. So, we can write this above inequality as
  \begin{align*}
    0\geq\sum_{i=1}^m\varphi_i(x,v_i)-\inner{x}{v_i}.
  \end{align*}
  Since all terms of this sum are non-negative, each one is equal to
  $0$ and the conclusion follows from item~\ref{it:f3} of
  Lemma~\ref{lm:fitz}.
\end{proof}

\end{document}